\documentclass[11pt, a4paper]{amsart}

\usepackage[T1]{fontenc}
\usepackage[utf8]{inputenc}
\usepackage{lmodern}
\usepackage{mathtools,amssymb}
\usepackage{enumitem}
\usepackage{microtype}
\usepackage[colorlinks=true,linkcolor=blue,citecolor=blue,urlcolor=blue]{hyperref}
\usepackage[a4paper,margin=1in]{geometry}

\theoremstyle{definition}
\newtheorem{definition}{Definition}[section]

\theoremstyle{plain}
\newtheorem{theorem}[definition]{Theorem}

\newtheorem{lemma}[definition]{Lemma}

\theoremstyle{remark}
\newtheorem*{remark}{Remark}

\title
[A c.e.\ \(tt\)-Degree Without c.e.\ Irreducible \(m\)-Degrees]
{A Computably Enumerable \(tt\)-Degree\\
Without Computably Enumerable Irreducible \(m\)-Degrees}

\author{Patrizio Cintioli}
\address{Mathematics Division, School of Science and Technology, University of Camerino, Italy}
\email{patrizio.cintioli@unicam.it}


\begin{document}

\maketitle

\begin{abstract}
In this paper, we provide a negative solution to Problem 3 formulated by P.~Odifreddi in his survey articles \textit{``Strong Reducibilities''} (1981) and \textit{``Reducibilities''} (1999). The problem asks whether every computably enumerable (c.e.) $tt$-degree contains a c.e.\ \textit{irreducible} $m$-degree (i.e., an $m$-degree consisting of only one $1$-degree). We answer this question in the negative by proving the existence of a c.e.\ $tt$-degree that does not contain any c.e.\ irreducible $m$-degree. Our proof relies on the structural properties of c.e.\ semirecursive sets with a rigid complement, originally constructed by A.~N.~Degtev. We show that the unique c.e.\ $m$-degree contained within the $tt$-degree of such a set consists of simple sets, which cannot be cylinders, and therefore necessarily splits into multiple $1$-degrees. Furthermore, our result demonstrates that a classical 1969 theorem by C.~G.~Jockusch Jr.---which guarantees the existence of an irreducible $m$-degree within every c.e.\ $tt$-degree---is strictly optimal and cannot be generally strengthened to require such an $m$-degree to be computably enumerable.
\end{abstract}

\section{Introduction}

The classification of computably enumerable (c.e.) sets and the study of relative computability rely heavily on the notion of reducibilities. While Turing reducibility ($\le_T$) captures the general concept of algorithmic relative computability, stronger reducibilities---such as truth-table ($\le_{tt}$), many-one ($\le_m$), and one-one ($\le_1$) reducibilities---provide a finer framework for analyzing the structural properties of c.e.\ sets. The relationships between these degree structures have been extensively studied since the foundational work of E.~L.~Post~\cite{Post1944}.

A recurring theme in Computability Theory is understanding how a coarser degree splits into finer degrees. A well-known concept in this context is that of an \textit{irreducible $m$-degree}, defined as an $m$-degree that contains exactly one $1$-degree~\cite{Downey1993}.
In 1969, C.~G.~Jockusch Jr.~\cite{Jockusch1969} and, independently, R.~I.~Soare~\cite{Soare1969} investigated the inner structure of these degrees. They demonstrated that every c.e.\ Turing degree contains a c.e.\ irreducible $m$-degree. Furthermore, Jockusch~\cite{Jockusch1969} established that every c.e.\ $tt$-degree contains an irreducible $m$-degree. He achieved this by considering the $tt$-cylinder of a c.e.\ set. However, because evaluating a truth-table condition on a $tt$-cylinder generally requires negative queries to the oracle, the resulting $tt$-cylinder of a c.e.\ non-computable set is fundamentally not computably enumerable. Thus, Jockusch's method could not guarantee the existence of an irreducible $m$-degree that is \textit{simultaneously} computably enumerable within an arbitrary c.e.\ $tt$-degree.
This subtle but crucial gap left open the question of whether both properties could be simultaneously satisfied. P.~Odifreddi formulated the question explicitly, which became known as Problem 3 in his renowned lists of open problems regarding strong reducibilities (see \textit{``Strong Reducibilities''} (1981)~\cite{Odifreddi1981} and \textit{``Reducibilities''} (1999)~\cite{Odifreddi1999}):

\medskip
\noindent\textbf{Problem 3}~\cite{Odifreddi1981, Odifreddi1999}. \textit{``Does every r.e.\ $tt$-degree contain an r.e.\ $m$-degree consisting of only one $1$-degree?''}

\smallskip
\noindent\textit{(Note: While the original quote uses the historical term ``recursively enumerable'' (r.e.), throughout this paper we adopt the modern standard terminology ``computably enumerable'' (c.e.). Additionally, an $m$-degree consisting of only one $1$-degree is commonly referred to in the literature as an \textbf{irreducible} $m$-degree~\cite{Downey1993}.)}
\medskip

In this paper, we provide a definitive negative solution to Odifreddi's Problem 3. We prove that the combination of properties requested by the problem cannot be universally achieved. To establish this, we analyze a c.e.\ semirecursive set with a rigid complement, first constructed by A.~N.~Degtev~\cite{Degtev1972}. Degtev proved that the c.e.\ $tt$-degree of such a set undergoes a drastic structural collapse, containing exactly one c.e.\ $m$-degree~\cite{Degtev1973}. By formally demonstrating that a c.e.\ set with a rigid complement is necessarily a simple set, and proving that a simple set can never be a cylinder, we deduce that this unique c.e.\ $m$-degree must split into multiple $1$-degrees (i.e., it is not irreducible). Consequently, the given c.e.\ $tt$-degree contains no c.e.\ irreducible $m$-degree.

Beyond resolving Odifreddi's Problem 3, this result reveals a profound historical corollary: Jockusch's 1969 theorem for $tt$-degrees is strictly optimal. The existence of an irreducible $m$-degree within a c.e.\ $tt$-degree is guaranteed, but one cannot generally demand that such an $m$-degree be computably enumerable.

Related work on the internal structure of c.e.\ $tt$-degrees was carried
out by Cholak and Downey~\cite{CholakDowney1994}. They proved that every
finite distributive lattice can be realized as the structure of the c.e.\
$m$-degrees contained in a c.e.\ $tt$-degree. Their result illustrates the
considerable variety of possible configurations of c.e.\ $m$-degrees
inside c.e.\ $tt$-degrees.

The paper is organized as follows. In Section~\ref{sec:preliminary}, we review the preliminary definitions of cylinders, simple sets, and rigid complements. In Section~\ref{sec:preparatory_lemmas}, we develop the necessary preparatory lemmas characterizing the relationships between these classes. Finally, in Section~\ref{sec:main_theorem}, we combine these lemmas with Degtev's structural results to prove our Main Theorem.

\section{Preliminary Definitions}\label{sec:preliminary}

We assume the reader is familiar with the standard notation and basic concepts of Computability Theory. For comprehensive background material, undefined notions, and an in-depth treatment of computably enumerable sets and strong reducibilities, we refer the reader to the classic monographs by Odifreddi~\cite{Odifreddi1989, Odifreddi1999_Vol2}, Rogers~\cite{Rogers1967}, and Soare~\cite{Soare1987}.

Throughout this paper, $\omega = \{0, 1, \ldots\}$ denotes the set of natural numbers. Given two sets $X$ and $Y$, their symmetric difference is denoted by $X \mathbin{\triangle} Y$ and is defined as $(X \setminus Y) \cup (Y \setminus X)$.

We briefly recall the standard strong reducibilities. For $X, Y \subseteq \omega$, $X$ is \textit{many-one reducible} to $Y$ ($X \le_m Y$) if there is a total computable function $f$ such that $x \in X \iff f(x) \in Y$ for all $x \in \omega$. If $f$ can be chosen to be injective, $X$ is \textit{one-one reducible} to $Y$ ($X \le_1 Y$). Furthermore, $X$ is \textit{truth-table reducible} to $Y$ ($X \le_{tt} Y$) if there exists a Turing reduction from $X$ to $Y$ that halts on all inputs regardless of the oracle. For any reducibility $r \in \{1, m, tt\}$, the equivalence relation $X \equiv_r Y$ (meaning $X \le_r Y$ and $Y \le_r X$) partitions the power set of $\omega$ into equivalence classes called \textbf{$r$-degrees}. An $r$-degree is called \textit{computably enumerable} (c.e.) if it contains at least one c.e.\ set. An $m$-degree is called \textbf{irreducible} if it contains exactly one $1$-degree.

\begin{definition}[Cylinder]\label{def:cylinder}
A set $C \subseteq \omega$ is called a \textbf{cylinder} if there exists a set $A$ such that $C \equiv_1 A \times \omega$, where $A \times \omega = \{\langle x, n \rangle : x \in A, n \in \omega\}$ and $\langle \cdot, \cdot \rangle$ is a standard bijective pairing function. It is a well-known fact that for any set $A$, the relation $A \equiv_m A \times \omega$ always holds.
\end{definition}

\begin{definition}[Simple Set]\label{def:simple_set}
A c.e.\ set $S$ is called \textbf{simple} if its complement $\overline{S}$ is infinite but contains no infinite c.e.\ subset.
\end{definition}

\begin{definition}[$pm$-reducibility and Rigid Complement]\label{def:rigid_complement}
Given two sets $X, Y \subseteq \omega$, $X$ is said to be \textbf{$pm$-reducible} to $Y$ ($X \le_{pm} Y$) if there exists a partial computable function $\psi$ such that for all $x \in \omega$: $x \in X \iff \psi(x)\!\downarrow \text{ and } \psi(x) \in Y$. Two sets are $pm$-equivalent ($X \equiv_{pm} Y$) if the reducibility holds in both directions.

A c.e.\ set $A$ has a \textbf{rigid complement} if its complement $\overline{A}$ is infinite and, for any two arbitrary subsets $X, Y \subseteq \overline{A}$, the equivalence $X \equiv_{pm} Y$ implies that their symmetric difference $X \mathbin{\triangle} Y$ is finite.
\end{definition}

\section{Preparatory Lemmas}\label{sec:preparatory_lemmas}

In this section, we establish three structural lemmas that form the logical foundation of our main result. The overarching strategy is to demonstrate a strict incompatibility between the structural requirements of an irreducible $m$-degree and the properties of c.e.\ sets with a rigid complement.

To this end, we first characterize irreducible $m$-degrees in terms of cylinders (Lemma~\ref{lemma:cylinders}). Next, we show that a simple set can never be a cylinder (Lemma~\ref{lemma:cylinder_not_simple}). Finally, we bridge these concepts by proving that the strong condition of having a rigid complement strictly forces a c.e.\ set to be simple (Lemma~\ref{lemma:rigid_implies_simple}). Together, these results provide the exact mathematical machinery needed to prove the Main Theorem.

\begin{lemma} \label{lemma:cylinders}
An $m$-degree $\mathbf{m}$ is irreducible if and only if every set $A \in \mathbf{m}$ is a cylinder.
\end{lemma}
\begin{proof}
Let $A \in \mathbf{m}$. Since universally $A \equiv_m A \times \omega$, the set $A \times \omega$ necessarily belongs to $\mathbf{m}$. In order for the $m$-degree $\mathbf{m}$ to be irreducible (i.e., to coincide with a single $1$-degree), all its elements must be mutually $1$-equivalent. Consequently, $A \equiv_1 A \times \omega$ must hold, which is the exact definition of a cylinder.

The converse implication follows from Myhill's Isomorphism Theorem: two $m$-equivalent cylinders are always $1$-equivalent (and computably isomorphic).
\end{proof}

\begin{lemma} \label{lemma:cylinder_not_simple}
No cylinder with a non-empty complement can be a simple set.
\end{lemma}
\begin{proof}
Let $C$ be a cylinder such that $\overline{C} \neq \emptyset$. By definition, there exists a set $A \neq \omega$ such that $C \equiv_1 A \times \omega$. 
By Myhill's Isomorphism Theorem, two $1$-equivalent sets are computably isomorphic. Therefore, there exists a computable permutation $h \colon \omega \to \omega$ (a total computable bijection with a computable inverse) such that $h(A \times \omega) = C$. Being a bijection, $h$ maps the complement exactly onto the complement: $h(\overline{A \times \omega}) = \overline{C}$.

Since $A \neq \omega$, there exists at least one element $a \notin A$. The set $R = \{a\} \times \omega$ is clearly an infinite and computable (hence c.e.) set, and it is entirely contained within $\overline{A \times \omega} = \overline{A} \times \omega$.

Since $h$ is computable and injective, its image $h(R)$ is also an infinite c.e.\ set. Given that $h(R) \subseteq \overline{C}$, we have identified an infinite c.e.\ set within the complement of $C$. By definition, this prevents $C$ from being a simple set.
\end{proof}

\begin{lemma} \label{lemma:rigid_implies_simple}
If a c.e.\ set $A$ has a rigid complement, then $A$ is a simple set.
\end{lemma}

\begin{proof}
By the very definition of a rigid complement (Definition~\ref{def:rigid_complement}), the complement $\overline{A}$ is infinite.

We proceed by contradiction, assuming that $A$ is not simple. Then $\overline{A}$ must contain an infinite c.e.\ subset $W$.

Since $W$ is c.e.\ and infinite, we can effectively enumerate and partition it into two disjoint infinite c.e.\ sets, $W_1$ and $W_2$ (for instance, given an effective enumeration of $W$ without repetitions, we can place the elements enumerated at even steps into $W_1$ and those at odd steps into $W_2$). Both are clearly subsets of $\overline{A}$.

We formally prove that $W_1 \le_{pm} W_2$. Let us fix an arbitrary element $w_2 \in W_2$ (which is always possible since $W_2$ is infinite and thus non-empty).
We define the partial function $\psi(x)$ as follows: we start the enumeration of $W_1$; if and when $x$ appears in the enumeration, the function halts and outputs $\psi(x) = w_2$. Otherwise, the search continues indefinitely ($\psi(x)\!\uparrow$).
Since $W_1$ is c.e., this procedure defines a valid algorithm, making $\psi$ a partial computable function.

It is immediately verified that:
\begin{itemize}
    \item if $x \in W_1$, then $\psi(x)\!\downarrow$ and $\psi(x) = w_2 \in W_2$;
    \item if $x \notin W_1$, the function diverges, hence $\psi(x)\!\uparrow$.
\end{itemize}
This perfectly satisfies the condition $x \in W_1 \iff \psi(x)\!\downarrow \land \; \psi(x) \in W_2$, proving that $W_1 \le_{pm} W_2$. 
By symmetry, $W_2 \le_{pm} W_1$ also holds, demonstrating that $W_1 \equiv_{pm} W_2$.

By the rigidity of the complement of $A$, any two subsets of $\overline{A}$ that are $pm$-equivalent must have a finite symmetric difference. However, since $W_1$ and $W_2$ are disjoint by construction, their symmetric difference is their union $W_1 \mathbin{\triangle} W_2 = W_1 \cup W_2 = W$. Since $W$ is infinite, we reach a clear contradiction.

It follows that the initial assumption was false: $\overline{A}$ cannot contain infinite c.e.\ sets, therefore $A$ is a simple set.
\end{proof}

\section{Main Theorem}\label{sec:main_theorem}

We are now ready to state and prove our main result. The proof synthesizes the structural properties established in the previous section with a powerful theorem concerning the collapse of strong degrees discovered by A.~N.~Degtev. By isolating a specific c.e.\ $tt$-degree in which the internal hierarchy of c.e.\ $m$-degrees collapses to a single equivalence class, we provide a definitive negative solution to Odifreddi's Problem 3.

\begin{theorem}[Solution to Problem 3]
There exists a computably enumerable $tt$-degree that does not contain any computably enumerable irreducible $m$-degree.
\end{theorem}

\begin{proof}
We rely on a sequence of celebrated results by A.~N.~Degtev. First, Degtev~\cite{Degtev1972} constructively proved the existence of a semirecursive c.e.\ set $A$ with a rigid complement. Subsequently, in his 1973 paper~\cite[Corollary 2.2]{Degtev1973} (also mentioned by Odifreddi~\cite{Odifreddi1999} in the context of Problem 9), he demonstrated a fundamental property of such sets: the c.e.\ $tt$-degree of $A$ (which we denote by $\mathbf{a}_{tt}$) contains exactly one c.e.\ $m$-degree (which we denote by $\mathbf{m}_A$, the degree to which the set $A$ obviously belongs).

Let us analyze the structural properties of this unique $m$-degree $\mathbf{m}_A$:
\begin{enumerate}
    \item Since the set $A$ has a rigid complement, by Lemma~\ref{lemma:rigid_implies_simple}, $A$ is a \textit{simple set}.
    \item Since $A$ is simple, by Lemma~\ref{lemma:cylinder_not_simple}, $A$ \textit{is not a cylinder}.
    \item Since the set $A \in \mathbf{m}_A$ is not a cylinder, it follows by Lemma~\ref{lemma:cylinders} that the $m$-degree $\mathbf{m}_A$ \textit{is not irreducible}. 
\end{enumerate}

By Degtev's theorem, $\mathbf{m}_A$ is the \textit{only} available c.e.\ $m$-degree within the entire $tt$-degree $\mathbf{a}_{tt}$. Having just shown that it is not irreducible, we unequivocally deduce that no c.e.\ irreducible $m$-degree exists in $\mathbf{a}_{tt}$.

This provides a rigorous proof and a negative answer to Odifreddi's Problem 3.
\end{proof}

\begin{remark}
Our negative solution to Odifreddi's Problem 3 provides an interesting corollary regarding a classical result by C.~G.~Jockusch Jr.~\cite{Jockusch1969}. Jockusch proved that every c.e.\ $tt$-degree contains an irreducible $m$-degree. Our result demonstrates that Jockusch's theorem is strictly optimal and cannot be generally strengthened to require such an $m$-degree to be computably enumerable. Indeed, within the specific c.e.\ $tt$-degree $\mathbf{a}_{tt}$ considered in our proof, the irreducible $m$-degree guaranteed by Jockusch's theorem necessarily exists, but it cannot be computably enumerable, since the unique c.e.\ $m$-degree $\mathbf{m}_A$ available in $\mathbf{a}_{tt}$ does not possess this property.
\end{remark}

\section*{Acknowledgments}

The main result was discovered by Gemini Deep Think (Google DeepMind). The author has fully reworked and verified all arguments and takes full responsibility for the mathematical content of the paper.

\end{document}